\documentclass[11pt]{amsart}
\usepackage{amsmath,amscd}
\usepackage{amsfonts}
\usepackage{amssymb}

\newcommand{\QQ}{{\mathbb Q}}
\newcommand{\FF}{{\mathbb F}}
\newcommand{\ZZ}{{\mathbb Z}}
\newcommand{\CC}{{\mathbb C}}
\newcommand{\Gal}{\mathrm{Gal}}
\newcommand{\Tr}{\mathrm{Tr}}

\newcommand{\cE}{{\mathcal E}}
\newcommand{\cL}{{\mathcal L}}

\newcommand{\bP}{{\bf P}}

\newcommand{\bG}{{\bf G}}
\newcommand{\bH}{{\bf H}}

\newcommand{\bT}{{\bf T}}
\newcommand{\bB}{{\bf B}}
\newcommand{\bU}{{\bf U}}
\newcommand{\bL}{{\bf L}}

\newcommand{\ex}{\mathrm{e}}

\newcommand{\Nr}{\mathrm{N}}
\newcommand{\Ind}{\mathrm{Ind}}

\newtheorem{theorem}{Theorem}[section]
\newtheorem{lemma}{Lemma}[section]
\newtheorem{proposition}{Proposition}[section]
\newtheorem{corollary}{Corollary}[section]

\newtheorem*{thm}{Theorem}

\def\adots{\mathinner{\mkern2mu\raise0pt\hbox{.}  
\mkern2mu\raise4pt\hbox{.}\mkern1mu
\raise7pt\vbox{\kern7pt\hbox{.}}\mkern1mu}}

\numberwithin{equation}{section}

\begin{document}

\bibliographystyle{ieeetr}

\title[Galois group action and Jordan decomposition]
{Galois group action and Jordan decomposition of characters of finite reductive groups with connected center}
\author{Bhama Srinivasan}
  \address{Department of Mathematics, Statistics, and Computer Science (MC 249)\\
           University of Illinois at Chicago\\
           851 South Morgan Street\\
           Chicago, IL  60680-7045}
  \email{srinivas@uic.edu}
\author{C. Ryan Vinroot}
  \address{Department of Mathematics\\
           College of William and Mary\\
           P. O. Box 8795\\
           Williamsburg, VA  23187-8795}
   \email{vinroot@math.wm.edu}

\begin{abstract}  Let $\bG$ be a connected reductive group with connected center defined over $\FF_q$, with Frobenius morphism $F$.  Given an irreducible complex character $\chi$ of $\bG^F$ with its Jordan decomposition, and a Galois automorphism $\sigma \in \Gal(\overline{\QQ}/\QQ)$, we give the Jordan decomposition of the image ${^\sigma \chi}$ of $\chi$ under the action of $\sigma$ on its character values.
\end{abstract}

\maketitle

\section{Introduction}

If $G$ is a finite group, the problem of understanding the action of the absolute Galois group on the irreducible characters of $G$ is a natural one.  The problem also has useful applications, an interesting example being a conjecture of G. Navarro \cite{Na04} which is a refinement of the McKay conjecture to take into account the Galois action on characters.  In particular, it is an important problem to understand the action of the Galois group on the irreducible characters of finite groups of Lie type, see \cite{ScTa18} for example, where the conjecture of Navarro is checked to hold for certain groups of Lie type.

In this paper, we describe the action of the Galois group on the irreducible characters of finite reductive groups with connected center, in terms of the Jordan decomposition of characters.  This is a generalization of our results from a previous paper \cite{SrVi15} where we accomplish this for the action of complex conjugation, and so describe the real-valued characters in terms of the Jordan decomposition.  Our main result may be stated as follows.
 
\begin{thm} [Theorem \ref{MainThm}]  Let $\bG$ be a connected reductive group with connected center, defined over $\FF_q$ with Frobenius morphism $F$.  Let $m$ be the exponent of $\bG^F$, and $\sigma \in \Gal(\QQ(\zeta_m)/\QQ)$ where $\zeta_m$ is a primitive $m$th root of unity, with $\sigma(\zeta_m) = \zeta_m^r$ where $r \in \ZZ$ and $(r,m)=1$.

Let $\chi$ be an irreducible complex character of $\bG^F$ with Jordan decomposition $(s_0, \nu)$, where $s_0 \in \bG^{*F^*}$ is a semisimple element in a dual group and $\nu$ is a unipotent character of $C_{\bG^*}(s_0)^{F^*}$.  Then ${^\sigma \chi}$ has Jordan decomposition $(s_0^r, {^\sigma \nu})$.
\end{thm}

We also give in Corollary \ref{MainCor} criteria to determine the field of character values of an irreducible character based on Jordan decomposition, and in Corollary \ref{LastCor} we give a particularly simple condition which implies a character is rational-valued.  These results reduce the question of the image of an irreducible character of $\bG^F$ under a Galois automorphism to understanding conjugacy of semisimple elements, which is well understood, and understanding the fields of character values of, and the action of group automorphisms on unipotent characters, both of which are well-studied problems \cite{Ge03, Lu02, Ma17}.

The organization of this paper is as follows.  In Section \ref{Prelims}, we establish notation for reductive groups, and in Proposition \ref{expon} we prove that a finite reductive group $\bG^F$ and its dual $\bG^{*F^*}$ have the same exponent.  If this common exponent is $m$, this allows us to work with automorphisms from $\Gal(\QQ(\zeta_m)/\QQ)$ which act on all irreducible characters of $\bG^F$, $\bG^{*F^*}$, and all of their subgroups.  While we could just as easily work with the Galois group of $\Gal(\QQ(\zeta_n)/\QQ)$, where $n = |\bG^F| = |\bG^{*F^*}|$, it is nicer to work with this more refined result, and Proposition \ref{expon} may also be of independent interest.

In Section \ref{Chars}, we give the basic character theory of finite reductive groups, including Lusztig series and unipotent characters, and we prove several lemmas needed for the main result.  We introduce the Jordan decomposition of characters in Section \ref{JordanDecomp}, including the crucial result of Digne and Michel in Theorem \ref{UniqueJord} that there exists a unique Jordan decomposition map with respect to a list of properties when the center $Z(\bG)$ is connected.  In Proposition \ref{UniqueJord}, we are able to slightly strengthen one property of Theorem \ref{UniqueJord} regarding unipotent characters.  Finally, our main results are proved in Section \ref{MainResults}.
\\
\\
\noindent{\bf Acknowledgements.  } The authors would like to thank Paul Fong, Meinolf Geck, Alan Roche, Amanda Schaeffer Fry, Jay Taylor, Donna Testerman, and Pham Huu Tiep for helpful communication over the course of working on this paper.  The authors also thank Gunter Malle for pointing out an inaccuracy in a lemma in an earlier draft of this paper.  The second-named author was supported in part by a grant from the Simons Foundation, Award \#280496.

\section{Preliminaries on Reductive Groups} \label{Prelims}

In this paper we follow the notation of \cite[Section 2]{SrVi15}, which we now recall.  Let $\bG$ be a connected reductive group defined over a finite field $\FF_q$ (with $p = \mathrm{char}(\FF_q)$ and fixed algebraic closure $\overline{\FF}_q$), with corresponding Frobenius morphism $F: \bG \rightarrow \bG$.  For any $F$-stable subgroup $\bG_1$ of $\bG$, $\bG_1^F$ will denote the group of $F$-fixed elements of $\bG_1$.  For any $g \in \bG$, we write ${^g \bG_1} = g \bG_1 g^{-1}$.

Fix a maximally split $F$-stable torus $\bT$ of $\bG$, contained in a fixed $F$-stable Borel subgroup $\bB$ of $\bG$.  Through the root system associated with $\bT$, we define a dual reductive group $\bG^*$ with dual Frobenius morphism $F^*$, and with $F^*$-stable maximal torus $\bT^*$ dual to $\bT$, contained in the $F^*$-stable Borel $\bB^*$ of $\bG^*$.  Define the Weyl group $W = \Nr_{\bG}(\bT)/\bT$, and the dual Weyl group $W^* = \Nr_{\bG^*}(\bT^*)/\bT^*$.  There is a natural isomorphism $\delta: W \rightarrow W^*$ \cite[Sec. 4.2]{Ca85}, and a corresponding anti-isomorphism, $w \mapsto w^* = \delta(w)^{-1}$.  The isomorphism $\delta$ restricts to an isomorphism between $W^F = \Nr_{\bG}(\bT)^F/\bT^F$ and $(W^*)^{F^*}$ \cite[Sec. 4.4]{Ca85}.  Let $l$ denote the standard length function on these Weyl groups.

Recall that the $\bG^F$-conjugacy classes of $F$-stable maximal tori in $\bG$ may be classified by $F$-conjugacy classes of $W$ as follows \cite[Sec. 8.2]{CaEn04}.  For any $F$-stable torus $\bT'$ in $\bG$, we have $\bT' = {^g \bT}$ for some $g \in \bG$.  Then $g^{-1} F(g) \in \Nr_{\bG}(\bT)$ and $w = g^{-1} F(g) \bT \in W$.  The $\bG^F$-conjugacy class of $\bT'$ then corresponds to the $F$-conjugacy class of $w$ in $W$.  Then we have $\bT'^F = g (\bT^{wF})g^{-1}$.  We say that $\bT'$ is an $F$-stable torus of $\bG$ of type $w$ (noting that the reference torus $\bT$ is fixed).  Since $\bT'^F$ and $\bT^{wF}$ are isomorphic, we work with $\bT^{wF}$ instead of $\bT'^F$.

Similar to the case of tori, the $\bG^F$-conjugacy classes of $F$-stable Levi subgroups are classified as follows.  Let $\bL$ be a Levi subgroup of a standard parabolic $\bP$, and given $w \in W$, let $\dot{w}$ denote an element in $\Nr_{\bG}(\bT)$ which reduces to $w$ in $W$.  Then any Levi subgroup of $\bG^F$ is isomorphic to $\bL^{\dot{w}F}$ for some $w \in W$, and we work with $\bL^{\dot{w}F}$ instead of the Levi subgroup of $\bG^F$.  For precise statements, see \cite[Sec. 8.2]{CaEn04}, \cite[Prop. 4.3]{DiMi90}, or \cite[Prop. 26.2]{MaTe11}.

If $\bT'$ is an $F$-stable torus of $\bG$ which is type $w$, then the $F^*$-stable maximal torus of type $F^*(w^*)$ in $\bG^*$ (with respect to $\bT^*$) is the dual torus $(\bT')^*$ of $\bT'$.  It follows that the finite tori $\bT^{wF}$ and $\bT^{*(wF)^*}$ are in duality, and there is an isomorphism, which we fix as in \cite[Sec. 8.2]{CaEn04}, between $\bT^{*F^*}$ and the group of characters $\hat{\bT}^F$ of $\bT^F$,
\begin{align*}
\bT^{*F^*} & \longleftrightarrow   \hat{\bT}^F \\
s & \longmapsto \theta = \hat{s}.
\end{align*}
Since $\bT^{wF}$ is in duality with $\bT^{*(wF)^*}$, then we may replace $F$ with $wF$, and $F^*$ with $(wF)^*$ in the correspondence above.   In particular, if $s \in \bT^{*(wF)^*}$ for some $w \in W$, then we denote by $\hat{s}$ the corresponding character in $\hat{\bT}^{wF}$.

Consider any semisimple element $s_0 \in \bG^{*F^*}$.  Then $s_0$ is contained in an $F^*$-stable maximal torus of $\bG^*$, and as above, we have $s_0 \in g(\bT^{*(wF)^*})g^{-1}$ for some $w \in W$ and $g \in \bG^*$.  We may correspond to $s_0$ (non-uniquely) the element $s = g^{-1} s_0  g\in \bT^*$, where $s$ is $(wF)^*$-fixed, and vice versa.  We then say that $s_0$ and $s$ are \emph{associated} semisimple elements, and we obtain that the $\bG^{*F^*}$-conjugacy class of $s_0$ is associated with the $W^*$-conjugacy class of $s$ through this correspondence.  Given any element $s \in \bT^*$, define $W_F(s)$ as
$$ W_F(s) = \{ w \in W \, \mid \, {^{(wF)^*} s} = s \}.$$
Then, the semisimple elements in $\bG^{*F^*}$ correspond to elements $s \in \bT^*$ such that $W_F(s)$ is nonempty.  Given such an $s \in \bT^*$, consider $C_{\bG^*}(s)$ and its Weyl group $W^*(s)$ relative to $\bT^*$, and define $W(s)$ to be the collection of elements $w \in W$ such that $w^* \in W^*(s)$.  Then, as in \cite[Section 2]{DiMi90}, we may write 
$$ W_F(s) = w_1 W(s),$$
where $w_1 \in W_F(s)$ is such that $\bT^{*(w_1 F)^*}$ is the maximally split torus inside of $C_{\bG^*}(s)^{(\dot{w}_1 F)^*}$, and also inside of $(C_{\bG^*}(s)^{\circ})^{(\dot{w}_1F )^*}$.  The maximal tori in $(C_{\bG^*}(s)^{\circ})^{(\dot{w}_1 F)^*}$ are then isomorphic to a torus of the form $\bT^{*(wF)^*}$, for $w \in W_F(s)$, by the same classification of maximal tori which we applied to $\bG^{*F^*}$.

We denote the exponent of a finite group $H$ as $\ex = \ex(H)$.  Thus $\ex(H)$ is the smallest positive integer $e$ such that $h^e = 1$ for all $h \in H$.  We have the following result.

\begin{proposition} \label{expon} For any connected reductive group $\bG$ defined over $\FF_q$, with Frobenius $F$, and dual $\bG^*$ with dual Frobenius $F^*$, we have $\ex(\bG^F) = \ex(\bG^{*F^*})$.
\end{proposition}
\begin{proof}  By the Jordan decomposition of elements, we have that any $g \in \bG^F$ can be written as $g = su=us$ where $s, u \in \bG^F$ with $s$ semisimple and $u$ unipotent.  If $p = \mathrm{char}(\FF_q)$, then $p$-power order elements in $\bG^F$ are exactly the unipotent elements, and elements with order prime to $p$ in $\bG^F$ are exactly the semisimple elements \cite[Theorem 2.5]{MaTe11}.  It follows that the exponent of $\bG^F$ is given by the product of the maximum order of unipotent elements in $\bG^F$ with the least common multiple of the orders of semisimple elements.  Thus in order to show $\bG^F$ and $\bG^{*F^*}$ have the same exponent, we must show that their maximal $p$-power order elements have the same order, and the least common multiple of the orders of their semisimple elements are the same.

Let $s_0 \in \bG^F$ be any semisimple element.  Then $s_0 \in g(\bT^{wF})g^{-1}$ for some $g \in \bG$ and $w \in W$.  We have $\bT^{wF}$ is in duality with $\bT^{*(wF)^*}$ in $\bG^*$, and these are isomorphic as finite groups.  For some $g_1 \in \bG^*$, $g_1 (\bT^{*(wF)^*}) g_1^{-1}$ is a torus in $\bG^{*F^*}$ containing an element of the same order as $s_0$.  It follows that the least common multiple of orders of semisimple elements in $\bG^F$ and $\bG^{*F^*}$ are equal.

The notion of a regular element in a semisimple algebraic group was introduced by R. Steinberg \cite{St65} and was studied by him and others.  Here, an element in $\bG$ is regular if its centralizer in $\bG$ has minimal dimension.  We now recall a proof given by D. Testerman \cite[p. 70, Proof of Corollary 0.5]{Te95} that when $\bG$ is a simple algebraic group, then amongst unipotent elements of $\bG$, regular unipotent elements have maximum order, and we note that this argument also holds when $\bG$ is a connected reductive group.  If $u \in \bG$ is regular unipotent of order $o(u)$, take some Borel subgroup $\bB_1$ of $\bG$ with unipotent radical $\bU$ such that $u \in \bU$.  If $x \in \bG$ is any other unipotent element, we have $gxg^{-1} \in \bU$ for some $g \in \bG$.  The $\bB_1$-orbit of $u$ under conjugation is then dense in $\bU$ by \cite[Theorem 5.2.1]{Ca85}, and all elements in this orbit have order $o(u)$.  However, if $x$ (and $gxg^{-1}$) has order greater than $o(u)$, then this dense orbit must intersect the nonempty open set $\{v \in \bU \, \mid \, v^{o(u)} \neq 1 \}$, a contradiction.  Since the regular unipotent class of $\bG$ intersects $\bG^F$ \cite[Proposition 5.1.7]{Ca85}, and $p$-power order elements are always unipotent, then the orders of the maximal $p$-power order elements of $\bG^F$ and $\bG$ are the same.

We first assume that $\bG$ is a simple algebraic group.  By \cite[Proposition 5.1.1]{Ca85}, for any connected reductive $\bG$ with center $Z(\bG)$, the natural homomorphism $\bG \rightarrow \bG/Z(\bG)$ induces a bijection between unipotent classes of $\bG$ and $\bG/Z(\bG)$, and in particular preserves orders of unipotent elements.  So when $\bG$ is simple, any other simple algebraic group isogenous to $\bG$ has maximal $p$-power order elements of the same order, and this only depends on root system type.  The only time $\bG$ is simple and $\bG^*$ has different root system type is when $\bG$ is type $B_m$ or $C_m$, and these types are dual to each other.  Testerman \cite[Corollary 0.5]{Te95} has computed the order of the maximal $p$-power order elements for all types.  For type $C_m$ or $B_m$, the maximal $p$-power order is the smallest $p$-power larger than $2m-1$ (see \cite[Proofs of Corollary 0.5 and Proposition 3.4]{Te95}).  It follows that when $\bG$ is a simple algebraic group, then the maximal $p$-power order elements of $\bG^F$ and $\bG^{*F^*}$ have the same order.

Next assume that $\bG$ is a semisimple algebraic group.  Write $\bG$ as the almost direct product of simple factors, $\bG = \bG_1 \cdots \bG_k$, where $[\bG_i, \bG_j]=1$ when $i \neq j$.  Then a maximal $p$-power order element of $\bG$ is the element of maximal $p$-power order of all the $\bG_i$.  If $\Phi$ is the root system of $\bG$, then $\Phi$ is an orthogonal disjoint union of $\Phi_i$, the root systems of $\bG_i$ \cite[Exercise 10.33]{MaTe11}.  It follows that $\bG^*$ has root system type which is an orthogonal disjoint union of the root system types of $\bG_i^*$.   From the case for simple algebraic groups, $\bG$ and $\bG^*$ have the same order of maximal $p$-power order elements, and this holds for any pair of semisimple algebraic groups with dual root system types.

Finally, when $\bG$ is reductive, write $\bG = [\bG, \bG] Z^{\circ}(\bG)$, where $[\bG, \bG]$ is semisimple.  Since no non-trivial unipotent elements of $\bG$ are central, then the regular unipotent elements of $\bG$ are contained in the semisimple factor.  Since $\bG$ and $[\bG, \bG]$ have the same root system type \cite[Corollary 8.1.9]{Sp98}, it follows that $[\bG, \bG]$ and $[\bG^*, \bG^*]$ have dual root system types.  From the semisimple algebraic group case, the orders of the maximal $p$-power elements are the same in $[\bG, \bG]$ and $[\bG^*, \bG^*]$, and so are the same in $\bG$ and $\bG^*$.  Thus $\bG^F$ and $\bG^{*F^*}$ have the same order maximal $p$-power order elements, and we have $\ex(\bG^F) = \ex(\bG^{*F^*})$.
\end{proof}

\section{Characters of Finite Reductive Groups} \label{Chars}

In this section we give some general theory and establish several preliminary results on the complex characters of finite reductive groups $\bG^F$.  We fix a prime $\ell$ which is distinct from $p = \mathrm{char}(\FF_q)$, we let $\QQ_{\ell}$ denote the $\ell$-adic numbers, and fix an algebraic closure $\overline{\QQ}_{\ell}$.  We fix an abstract isomorphism of fields $\CC \cong \overline{\QQ}_{\ell}$, so that our characters take values in $\overline{\QQ}_{\ell}$.  In particular we identify a fixed algebraic closure $\overline{\QQ}$ of the rationals with its image in $\overline{\QQ}_{\ell}$ under this isomorphism.  

For any finite group $H$, any complex character (or $\overline{\QQ}_{\ell}$-valued character) $\eta$ of $H$, and any $\sigma \in \Gal(\overline{\QQ}/\QQ)$, we define the character ${^\sigma \eta}$ by ${^\sigma \eta}(h) = \sigma(\eta(h))$.  If $\rho$ is the representation with character $\eta$, then ${^\sigma \rho}$ is the representation with character ${^\sigma \eta}$.  

\subsection{Lusztig induction} \label{Linduction}

Deligne and Lusztig \cite{dellusz} defined certain virtual representations of finite reductive groups $\bG^F$ through the $\ell$-adic cohomology with compact support associated to algebraic varieties over $\overline{\FF}_q$.  If $X$ is such a variety, we denote its $i$th $\ell$-adic cohomology space with compact support with coefficients in $\overline{\QQ}_{\ell}$ as
$H_c^i(X, \overline{\QQ}_{\ell})$, and then $H_c^*(X) = \sum_i (-1)^i H_c^i(X, \overline{\QQ}_{\ell})$ is a virtual $\overline{\QQ}_{\ell}$-vector space.

Let $\bL$ be an $F$-stable Levi subgroup of $\bG$ of a standard parabolic $\bP$ (as in Section \ref{Prelims}) with Levi decomposition $\bP = \bL \bU$.  If $\cL: \bG \rightarrow \bG$ is the Lang map, $\cL(g) = g^{-1} F(g)$, then $\cL^{-1}(\bU)$ is an algebraic variety over $\overline{\FF}_q$.  Then the virtual $\overline{\QQ}_{\ell}$-space $H_c^*(\cL^{-1}(\bU))$ may be taken to be a $(\bG^F \times \bL^F)$-bimodule.  If we identify $\bL^F$ with $\bL^{\dot{w}F}$ as in Section \ref{Prelims}, then we may regard this as a $(\bG^F \times \bL^{\dot{w}F})$-bimodule, and it is through this structure that one defines the \emph{Lusztig induction} functor $R_{\bL^{\dot{w}F}}^{\bG^F}$ which takes characters of $\bL^{\dot{w}F}$ to virtual characters of $\bG^F$, and which is Harish-Chandra induction when the parabolic $\bP$ is $F$-stable.  While the definition of Lusztig induction depends on the choice of parabolic $\bP$, it is known in all but very few cases that this functor is independent of this choice \cite{BoMi11, Tay18}.  We will need the following statement (see also \cite[Lemma 2.1]{ScVi18}).

\begin{lemma} \label{DLGalois} For any Levi subgroup $\bL^{\dot{w}F}$ of $\bG^F$, any character $\gamma$ of $\bL^{\dot{w}F}$, and any $\sigma \in \Gal(\overline{\QQ}/\QQ)$, we have $^\sigma R_{\bL^{\dot{w}F}}^{\bG^F} (\gamma)  = R_{\bL^{\dot{w}F}}^{\bG^F}({^\sigma \gamma})$.
\end{lemma}
\begin{proof}  By \cite[Proposition 11.2]{dmbook}, we have for any $g \in \bG^F$,
\begin{equation} \label{DLind}
\left(R_{\bL^{\dot{w}F}}^{\bG^F}(\gamma)\right)(g) = \frac{1}{|\bL^{\dot{w}F}|} \sum_{l \in \bL^{\dot{w}F}} \Tr((g,l)  \mid  H^*_c(\cL^{-1}(\bU))) \gamma(l^{-1}).
\end{equation}
By \cite[Corollary 10.6]{dmbook}, for example, every $\Tr((g,l)  \mid  H^*_c(\cL^{-1}(\bU)))$ is a rational integer, so is stable under the action of $\sigma$.  By applying $\sigma$ to both sides of \eqref{DLind}, we have
\begin{align*}
\left(^\sigma R_{\bL^{\dot{w}F}}^{\bG^F} (\gamma) \right)(g) & = \frac{1}{|\bL^{\dot{w}F}|} \sum_{l \in \bL^{\dot{w}F}} \Tr((g,l)  \mid  H^*_c(\cL^{-1}(\bU))) ({^\sigma \gamma}(l^{-1})) \\
& = \left(R_{\bL^{\dot{w}F}}^{\bG^F}({^\sigma \gamma})\right)(g),
\end{align*}
proving the claim.
\end{proof}

\subsection{Lusztig series} \label{Lseries}

If one takes a torus $\bT^{wF}$ for the Levi subgroup in Lusztig induction, and $\theta$ is an irreducible character of $\bT^{wF}$, then we get the \emph{Deligne-Lusztig} virtual character $R_{\bT^{wF}}^{\bG^F} (\theta)$, originally defined in \cite{dellusz}.  If $s \in \bT^*$ is semisimple such that $W_F(s)$ is nonempty, then the \emph{rational Lusztig series} of $\bG^F$ corresponding to $s$, denoted $\cE(\bG^F, s)$, is the set of irreducible characters $\chi$ of $\bG^F$ such that, for some $w \in W_F(s)$ we have
$$ \langle \chi, R_{\bT^{wF}}^{\bG^F}(\hat{s}) \rangle \neq 0,$$
where $\langle \cdot, \cdot \rangle$ denotes the standard inner product on class functions.  Given $s, t \in \bT^*$ with both $W_F(s)$ and $W_F(t)$ nonempty, then $\cE(\bG^F, s)$ and $\cE(\bG^F, t)$ are either disjoint or equal, and are equal precisely when $s$ and $t$ are $W^*$-conjugate.  If $s_0$ and $t_0$ are semisimple elements of $\bG^{*F^*}$ associated with $s$ and $t$, respectively, as in Section \ref{Prelims}, then this is equivalent to $s_0$ and $t_0$ being $\bG^{*F^*}$-conjugate (see \cite[Propositon 13.13]{dmbook} and its proof).  We will either denote the rational Lusztig series by $\cE(\bG^F, s)$ when it is parameterized by the $W^*$-class of some $s \in \bT^*$ with $W_F(s)$ nonempty, or by $\cE(\bG^F, s_0)$ when it is parameterized by the $\bG^{*F^*}$-class of some semisimple element $s_0 \in \bG^{*F^*}$.

One may further define the \emph{geometric Lusztig series}, which is parameterized by the $\bG$-conjugacy class of a semisimple element $s_0 \in \bG^F$, and contains the associated rational Lusztig series.  We only remark that when the centralizer $C_{\bG^*}(s_0)$ (or $C_{\bG^*}(s)$) is connected, then the geometric and rational Lusztig series coincide.

Recall an irreducible character $\chi$ of $\bG^F$ is \emph{cuspidal} if it does not appear in the truncation to any standard parabolic subgroup of $\bG^F$, see \cite[Section 9.1]{Ca85}.  The set of cuspidal characters in the Lusztig series $\cE(\bG^F, s)$ will be denoted by $\cE(\bG^F, s)^{\bullet}$.

By Proposition \ref{expon}, we have $\ex(\bG^F) = \ex(\bG^{*F^*})=m$, and so any irreducible character of $\bG^F$, $\bG^{*F^*}$, or any of their subgroups, takes values in $\QQ(\zeta_m)$ for $\zeta_m$ a primitive $m$th root of unity.  If $\sigma \in \Gal(\overline{\QQ}/\QQ)$,  we also denote by $\sigma$ its projection to $\Gal(\QQ(\zeta_m)/\QQ)$.  So $\sigma$ acting on $\QQ(\zeta_m)$ is generated by $\sigma(\zeta_m) = \zeta_m^r$ for some $r \in \ZZ$ with $(r,m) = 1$.  We will need the following result, which is also a special case of \cite[Lemma 3.4]{ScTa18}.

\begin{lemma} \label{Galseries}  Let $\chi$ be an irreducible character of $\bG^F$, $\sigma \in \Gal(\overline{\QQ}/\QQ)$, and $r \in \ZZ$ such that $\sigma(\zeta_m) = \zeta_m^r$.  Then we have $\chi \in \cE(\bG^F, s)$ if and only if ${^\sigma \chi} \in \cE(\bG^F, s^r)$.
\end{lemma}
\begin{proof}  We have $\chi \in \cE(\bG^F, s)$ if and only if 
$$ \langle \chi, R_{\bT^{wF}}^{\bG^F} (\hat{s})  \rangle= \langle {^\sigma \chi}, {^\sigma R}_{\bT^{wF}}^{\bG^F} (\hat{s}) \rangle \neq 0$$
for some $w \in W_F(s)$, where the first equality is obtained by the fact that the inner product is a rational integer, and so stable under $\sigma$, and by applying $\sigma$ to each term of the sum defining the inner product.

Each linear character $\hat{s}$ takes values in $m$th roots of unity, since $\ex(\bT^{wF})$ divides $m$, and since $s \mapsto \hat{s}$ is a homomorphism, we have ${^\sigma \hat{s}} = \hat{s}^r= \widehat{s^r}$.  From Lemma \ref{DLGalois}, we thus have
$${^\sigma R}_{\bT^{wF}}^{\bG^F} (\hat{s}) = R_{\bT^{wF}}^{\bG^F}(\widehat{s^r}).$$
Now $\chi \in \cE(\bG^F, s)$ if and only if
$$\langle {^\sigma \chi}, R_{\bT^{wF}}^{\bG^F}(\widehat{s^r}) \rangle \neq 0,$$
which is true exactly when ${^\sigma \chi} \in \cE(\bG^F, s^r)$.
\end{proof}

\subsection{Unipotent characters} \label{unipotent}

The unipotent characters of $\bG^F$ are those irreducible characters in the Lusztig series $\cE(\bG^F, 1)$.  The unipotent characters of $\bG^F$ may be viewed as \emph{generic} objects associated with $\bG^F$ (see \cite[Section 1B]{BrMaMi93}), and there exists a canonical labeling of unipotent characters with certain uniqueness properties by a result of Lusztig \cite{Lu15} (see also \cite[Section 4]{Ge17}).

In this section we consider the eigenvalues of the Frobenius map acting on certain algebraic varieties, which correspond to unipotent characters, as studied by Lusztig in \cite{Lu76,Lu78}.  In particular, for any $w \in W$, we let $X_w$ be the algebraic variety over $\overline{\FF}_q$ given by the set of all Borel subgroups $\bB$ of $\bG$ which are mapped to $F(\bB)$ by $w$ (that is, the Deligne-Lusztig variety).  We may thus consider the spaces $H^i_c(X_w, \overline{\QQ}_{\ell})$.  If $\delta$ is the smallest positive integer such that $F^{\delta}$ acts trivially on $W$, then as in \cite[Chapter 3]{Lu78} there is a natural action of $F^{\delta}$ on $H_c^i(X_w, \overline{\QQ}_{\ell})$.  By \cite[Corollary 3.9]{Lu78}, for any unipotent representation $\pi$ of $\bG^F$ with character $\chi$, there exists a $w \in W$, $i \geq 0$, and $\alpha \in \overline{\QQ}_{\ell}^{\times}$ such that $\alpha$ is an eigenvalue of $F^{\delta}$ acting on $H_c^i(X_w, \overline{\QQ}_{\ell})$, and $\pi$ is isomorphic to a $\bG^F$-submodule of the generalized $\alpha$-eigenspace of $F^{\delta}$ on $H_c^i(X_w, \overline{\QQ}_{\ell})$.  Then we say $\alpha$ is an eigenvalue of $F^{\delta}$ associated with $\chi$ (or $\pi$). 

As in \cite[Section 4.1]{GeMa03}, for any unipotent character $\chi$ of $\bG^F$, any eigenvalue $\alpha$ of $F^{\delta}$ corresponding to $\chi$ is uniquely determined by $\chi$ up to a factor of the form $q^{k \delta}$ for some integer $k$.  There is then a root of unity $\omega_{\chi}$ and a factor $\beta_{\chi} \in \{ 1, q^{\delta/2} \}$ such that $\alpha = \omega_{\chi} \beta_{\chi} q^{k \delta}$ for some non-negative integer $k$.  We will need the following property of these values, which is given in \cite[Lemma 4.3]{GeMa03}, where the first statement was first proved in \cite[Cor. III.3.4]{DiMi85}.  

\begin{lemma} \label{Prop2a1} Let $\chi$ be any unipotent character of $\bG^F$ and $\sigma \in 
\Gal(\overline{\QQ}/\QQ)$.  Then we have
$$ \sigma(\omega_{\chi} \beta_{\chi}) = \omega_{{^\sigma \chi}} \beta_{{^\sigma \chi}},$$
and $\omega_{\chi} \beta_{\chi}$ is contained in the field of character values of $\chi$.
\end{lemma}

Note that if $\chi$ is a unipotent character of $\bG^F$, then the fact that ${^\sigma \chi}$ is also a unipotent character follows from Lemma \ref{Galseries}.

When $\bG$ is a simple algebraic group, then the values $\omega_{\chi} \beta_{\chi}$ have all been calculated for unipotent characters $\chi$ of $\bG^F$, and follow from work done in \cite{Lu76, Lu78, Lu84, DiMi85, GeMa03}.  Based on those calculations, we make the following observation which we will apply in the next section.

\begin{lemma} \label{Prop2a} Let $\bG$ be a simple algebraic group of adjoint type defined over $\FF_q$ with Frobenius $F$, and let $(\bG^*, F^*)$ be the dual group.  Let $\chi$ be any unipotent character of $\bG^F$ and $\psi$ any unipotent character of $\bG^{*F^*}$.  Suppose $(\bG, F)$ is not of type $E_8$.  Then we have $\omega_{\chi} = \omega_{\psi}$ implies $\beta_{\chi} = \beta_{\psi}$.  If $(\bG, F)$ is of type $E_8$, we have $\omega_{\chi}=\omega_{\psi}$ and $\chi(1)=\psi(1)$ implies $\beta_{\chi} = \beta_{\psi}$.
\end{lemma}
\begin{proof}  Suppose that $(\bG, F)$ and $(\bH, F')$ are such that $\bG$ and $\bH$ are simple algebraic groups defined over $\FF_q$, and $f: (\bG, F) \rightarrow (\bH, F')$ is an isogeny.  It follows from \cite[Proposition 3.15]{Lu78} that the isogeny $f$ induces a natural bijection between the unipotent characters of $\bG^F$ and those of $\bH^{F'}$.  The Weyl groups of $\bG$ and $\bH$ may be identified through the isogeny (call it $W$) with $F$ and $F'$ acting on $W$ in a compatible way.  It follows from \cite[1.18]{Lu76} that the isogeny $f$ allows an identification of the varieties $X_w$ corresponding to $\bG$ and $\bH$, with compatibility of the actions of $F^{\delta}$ and $F'^{\delta}$ on $H_c^i(X_w, \overline{\QQ}_{\ell})$ for any $w \in W$, so that the eigenvalues from each action are the same.  Thus the isogeny $f$ preserves the eigenvalues of the Frobenius corresponding to unipotent characters of $\bG^F$ and $\bH^{F'}$.  We now assume that $\bG$ is a simple algebraic group of adjoint type, and we note the only time that $(\bG, F)$ and $(\bG^*, F^*)$ are not isogenous in this case are when $\bG$ is type $B_n$ or $C_n$, which are then dual to each other.

Next, it follows from \cite[Cor. II.3.4]{DiMi85} and \cite[Lemma 4.4]{GeMa03} that we must have $\beta_{\chi}=1$ for any unipotent character $\chi$ of $\bG^F$, unless $(\bG, F)$ is of type $E_7$, $E_8$, ${^2 E_6}$, or ${^2 A_n}$.  Since we always have $\beta_{\chi}=1$ in types other than these, the statement follows immediately in all other types.  We consider specific values in the remaining types.

By \cite[Theorem 3.34(ii)]{Lu78}, when $(\bG, F)$ is type ${^2 A_n}$ (so $\delta = 2$), for any unipotent character of $\bG^F$ we must either have $\omega_{\chi}=1$ and $\beta_{\chi}=1$, or $\omega_{\chi}=-1$ and $\beta_{\chi} = q$, and the statement follows in this case.  When $(\bG, F)$ is type $E_7$ (so $\delta=1$), it follows from \cite[Theorem 3.34(iv)]{Lu78} that the set of values taken by $\omega_{\chi} \beta_{\chi}$ in this case is $\{1, -1, \omega, \omega^2, iq^{1/2}, -iq^{1/2} \}$, where $\omega$ is a primitive cube root of unity.  Since the only time we have $\beta_{\chi} = q^{1/2}$ is if $\omega_{\chi}=\pm i$, and we never have $\omega_{\chi} = \pm i$ and $\beta_{\chi} = 1$, the statement follows in this case.  When $(\bG, F)$ is of type ${^2 E_6}$ (so $\delta = 2$), it follows from \cite[Table 1]{Lu78} along with \cite[Lemma 4.2 and Remark 4.9]{GeMa03} that the set of possible values of $\omega_{\chi} \beta_{\chi}$ is $\{1, \omega, \omega^2, -q\}$ with $\omega$ a primitive cube root of unity, and the desired statement follows.

Finally, we consider the case $\bG^F = E_8(q)$, and the unipotent characters of $E_8(q)$ as they are labeled in \cite[pgs. 484--488]{Ca85}.  The values for $\omega_{\chi} \beta_{\chi}$ follow from \cite[Table 1]{Lu78}, \cite[Chapter 11]{Lu84}, \cite[Cor II.3.4]{DiMi85}, along with \cite[Lemma 4.2]{GeMa03}, and are given by
$$ \{1, -1, i, -i, \rho, \rho^2, \rho^3, \rho^4, \omega, \omega^2, -\omega, -\omega^2, iq^{1/2}, -iq^{1/2} \},$$
where $\omega$ is a primitive cube root of unity, and $\rho$ is a primitive fifth root of unity.  Further, the above references give that the only time $\omega_{\chi} \beta_{\chi} = i$ is when $\chi$ is labeled by $E_8[i]$, and the only times when $\omega_{\chi} \beta_{\chi} = i q^{1/2}$ are for $\chi$ labeled by $E_7[\xi], 1$ or $E_7[\xi], \varepsilon$, and these all have distinct character degrees.  Similarly, the only time when $\omega_{\chi} \beta_{\chi} = -i$ is when $\chi$ is labeled by $E_8[-i]$, and the only times when $\omega_{\chi} \beta_{\chi} = -iq^{1/2}$ are when $\chi$ is labeled by $ E_7[-\xi], 1$ or $E_7[-\xi], \varepsilon$, and again these all have distinct character degrees.
\end{proof}

If $\bG$ is a direct product, say $\bG = \bG_1 \times \bG_2$, then the unipotent characters of $\bG^F$ are all of the form $\chi_1 \times \chi_2$, where $\chi_1$ and $\chi_2$ are unipotent characters of $\bG_1^{F}$ and $\bG_2^{F}$, respectively, which is in \cite[pg. 28]{Lu78}.  We have the following.

\begin{lemma} \label{Prop2a2} Let $\bG$ be a direct product, say $\bG = \prod_i \bG_i$, where each $\bG_i$ is $F$-stable.  For any unipotent character $\chi = \prod_i \chi_i$ of $\bG^F$, where $\chi_i$ is a unipotent character of $\bG_i^F$, we have $\omega_{\chi} = \prod_i \omega_{\chi_i}$, and $\beta_{\chi} = 1$ if $\prod_i \beta_{\chi_i}$ is an integer power of $q^{\delta}$, and $\beta_{\chi} = q^{\delta/2}$ otherwise.
\end{lemma}
\begin{proof} We consider the statement when $\bG = \bG_1 \times \bG_2$, and the general case follows.  Any Borel subgroup of $\bG$ is of the form $\bB_1 \times \bB_2$ with $\bB_i$ a Borel subgroup of $\bG_i$, and the Weyl group $W$ is also a direct product, $W = W_1 \times W_2$.  It follows that, given $w \in W$ with $w = (w_1, w_2)$, the Deligne-Lusztig variety $X_w$ is a direct product, $X_w = X_{w_1} \times X_{w_2}$.  We then have by the K\"{u}nneth formula \cite[Proposition 10.9(i)]{dmbook}
$$ H_c^i(X_w, \overline{\QQ}_{\ell}) \cong \bigoplus_{j_1 + j_2 = i} H_c^{j_1}(X_{w_1}, \overline{\QQ}_{\ell}) \otimes_{\overline{\QQ}_{\ell}} H_c^{j_2}(X_{w_2}, \overline{\QQ}_{\ell}).$$
Thus the eigenvalues of $F^{\delta}$ associated with $\chi$ are obtained as products of the eigenvalues of $F^{\delta}$ associated with $\chi_1$ and $\chi_2$, and the statement follows.
\end{proof}

\noindent {\bf Remark. }  If $\bG$ is a direct product, say $\bG= \prod_{i=1}^k \bG_i$, with $F(\bG_i) = \bG_{i+1}$ for $1 \leq i \leq k-1$, and $F(\bG_k)=\bG_1$, then it follows that $\bG^F \cong \bG_1^{F^k}$.  We can thus assume $\bG_i = \bG_1$ for each $i$, and $F$ cyclically permutes $k$ direct product copies of $\bG_1$, and so $F^k$ may be viewed as an endomorphism of $\bG_1$.  Then the Weyl group $W$ of $\bG$ is a direct product of $k$ copies of the Weyl group $W_1$ of $\bG_1$.  If $F^{\delta}$ is the least power of $F$ which acts trivially on $W$, then $F^{k\delta}$ is the least power of $F^k$ which acts trivially on $W_1$.  As in \cite[(1.18)]{Lu76} and \cite[Proof of Proposition 6.4]{DiMi90}, any $X_w$ for $w \in W$ is isomorphic to one of the form $(w_1, 1, \ldots, 1) \in W_1^k$, which is then isomorphic to $X_{w_1}$ associated with $\bG_1$ and $F^{k \delta}$, with compatible actions.  It follows that the eigenvalues of $F^{\delta}$ acting on $H_c^i(X_w, \overline{\QQ}_{\ell})$ are equal to those of $F^{k \delta}$ acting on $H_c^i(X_{w_1}, \overline{\QQ}_{\ell})$.  Note that the dual group in this case satisfies $\bG^{*F^*} \cong \bG_1^{*F^{*k}}$.

\subsection{Principal series} \label{Princ}

The unipotent characters in the principal series of $\bG^F$ are the constituents of $\mathrm{Ind}_{\bB^F}^{\bG^F}(\mathbf{1})$.  Associated with the module $\mathrm{Ind}_{\bB^F}^{\bG^F}(\bf{1})$ is the Hecke algebra
$$ \mathcal{H} = \mathcal{H}(\bG^F, \bB^F) = \mathrm{End}_{\mathbb{C}\bG^F} \left( \mathrm{Ind}_{\bB^F}^{\bG^F}(\bf{1}) \right).$$
Then we have $\mathcal{H} = e \mathbb{C}\bG^F e$, where $e \in \mathbb{C}\bG^F$ is the idempotent element
$$ e = \frac{1}{|\bB^F|} \sum_{b \in \bB^F} b.$$
There is a natural bijection between unipotent characters in the principal series and the irreducible characters of the Hecke algebra $\mathcal{H}$, which is defined as follows \cite[Theorem 11.25(ii)]{CuRe81}.  Given a unipotent character $\chi$ in the principal series of $\bG^F$, extend $\chi$ linearly to a character $\tilde{\chi}$ of $\mathbb{C}\bG^F$, and then restrict to $\mathcal{H}$ to obtain an irreducible character $\tilde{\chi}|_{\mathcal{H}}$ of the Hecke algebra.  

Now let $\sigma \in \mathrm{Aut}(\mathbb{C}/\mathbb{Q})$, and we consider the action of $\sigma$ on $\chi$ and its effect on the bijection with irreducible characters of the Hecke algebra.

\begin{lemma} \label{Prop2b} Given a unipotent character $\chi$ in the principal series of $\bG^F$, and $\sigma \in \mathrm{Aut}(\mathbb{C}/\mathbb{Q})$, we have
$$ \widetilde{{^\sigma \chi}}|_{\mathcal{H}} = \sigma \circ (\tilde{\chi}|_{\mathcal{H}}) \circ \sigma^{-1}.$$
\end{lemma}
\begin{proof}  Given an element $\sum_g \alpha_g g \in \mathbb{C} \bG^F$, the action of $\sigma$ is defined as 
$$ \sigma \left(\sum_g \alpha_g g \right) = \sum_g \sigma(\alpha_g) g.$$
Note that we then have $\sigma(e) = e$, and since $\mathcal{H} = e\mathbb{C} \bG^F e$, then $\sigma \circ (\tilde{\chi}|_{\mathcal{H}}) \circ \sigma^{-1}$ is a well-defined character of $\mathcal{H}$.  We compute
\begin{align*}
\widetilde{{^\sigma \chi}} \left(\sum_g \alpha_g g \right) &=  \sum_g \alpha_g \sigma(\chi(g)) \\
 & = \sigma \left( \sum_g \sigma^{-1}(\alpha_g) \chi(g) \right) \\
& = \sigma \left( \tilde{\chi} \left( \sum_g \sigma^{-1}(\alpha_g) g \right) \right) \\
& = (\sigma \circ \tilde{\chi} \circ \sigma^{-1}) \left( \sum_g \alpha_g g \right).
\end{align*}
Thus $\widetilde{{^\sigma \chi}} = \sigma \circ \tilde{\chi} \circ \sigma^{-1}$.  Since we also have 
$$ (\sigma \circ \tilde{\chi} \circ \sigma^{-1})|_{\mathcal{H}} = \sigma \circ (\tilde{\chi}|_{\mathcal{H}}) \circ \sigma^{-1},$$
the result follows.
\end{proof}

The Hecke algebra $\mathcal{H}(\bG^F, \bB^F)$ has a basis indexed by the fundamental generating set of simple reflections for the Weyl group $W^F$ (see \cite[Chapter 10]{Ca85}, for example).  There is also the Hecke algebra $\mathcal{H}^* = \mathcal{H}(\bG^{*F^*}, \bB^{*F^*})$, corresponding to the dual group $\bG^{*F^*}$, again with a basis indexed by the simple reflections of the Weyl group $W^{*F^*}$.  Through the isomorphism $\delta$ between $W^F$ and $W^{*F^*}$, we identify the Hecke algebras $\mathcal{H}$ and $\mathcal{H}^*$, and their irreducible characters.

\section{Jordan Decomposition of Characters} \label{JordanDecomp}

Given the connected reductive group $\bG$ and a rational Lusztig series $\cE(\bG^F, s)$, where $s \in \bT^*$ with $W_F(s)$ nonempty, a \emph{Jordan decomposition map} is a bijection $J_s^{\bG} = J_s$,
$$ J_s : \cE(\bG^F, s) \longrightarrow \cE(C_{\bG^*}(s)^{(\dot{w}_1 F)^*}, 1),$$
with the property that, for any $\chi \in \cE(\bG^F, s)$ and any $w \in W_F(s)$, we have
\begin{equation} \label{innprod}
\langle \chi, R_{\bT^{wF}}^{\bG^F}(\hat{s}) \rangle = \langle J_s(\chi), (-1)^{l(w_1)} R_{\bT^{*(wF)^*}}^{C_{\bG^*}(s)^{(\dot{w}_1 F)^*}} ({\bf 1}) \rangle.
\end{equation}
The Jordan decomposition map was proved to exist in the case that the center $Z(\bG)$ is connected by Lusztig \cite{Lu84}, and in the case that the center is disconnected by Lusztig \cite{Lu88} and by Digne and Michel \cite{DiMi90}.  In the case that $Z(\bG)$ is disconnected, then there are rational Lusztig series $\cE(\bG^F, s)$ such that $C_{\bG^*}(s)$ is disconnected.  In this case, one defines the unipotent characters in the set $\cE(C_{\bG^*}(s)^{(\dot{w}_1 F)^*}, 1)$ to be those characters which appear in the induction from unipotent characters of the group $(C_{\bG^*}(s)^{\circ})^{(\dot{w}_1 F)^*}$.

Lusztig found \cite{Lu84} that in many cases the Jordan decomposition map $J_s$ is completely determined by the property \eqref{innprod}, although this is not always true.  In the case that $Z(\bG)$ is connected, and so $C_{\bG^*}(s)$ is connected for any $s$, Digne and Michel \cite[Theorem 7.1]{DiMi90} found a list of properties which uniquely determines the Jordan decomposition map, which we now state.

\begin{theorem}[Digne and Michel, 1990] \label{UniqueJord}  Suppose that $Z(\bG)$ is connected.  Given any $s \in \bT^*$ such that $W_F(s)$ is nonempty, there exists a unique bijection
$$ J_s : \cE(\bG^F, s) \longrightarrow \cE(C_{\bG^*}(s)^{(\dot{w}_1 F)^*}, 1)$$
which satisfies the following conditions:
\begin{enumerate}
\item For any $\chi \in \cE(\bG^F, s)$, and any $w \in W_F(s)$, 
$$ \langle \chi, R_{\bT^{wF}}^{\bG^F}(\hat{s}) \rangle = \langle J_s(\chi), (-1)^{l(w_1)} R_{\bT^{*(wF)^*}}^{C_{\bG^*}(s)^{(\dot{w}_1 F)^*}} ({\bf 1}) \rangle.$$
\item If $s=1$ then:
\begin{itemize}
\item[(a)] The eigenvalues of $F^{\delta}$ associated to $\chi$ are equal, up to an integer power of $q^{\delta/2}$, to the eigenvalues of $F^{*\delta}$ associated to $J_1(\chi)$.
\item[(b)] If $\chi$ is in the principal series then $J_1(\chi)$ and $\chi$ correspond to the same character of the Hecke algebra.
\end{itemize}

\item If $z \in Z(\bG^{*F^*})$ is central, and $\chi \in \cE(\bG^F, s)$, then $J_{sz}(\chi \otimes \hat{z}) = J_s(\chi)$.
\item If $\bL$ is a standard Levi subgroup of $\bG$ such that $\bL^*$ contains $C_{\bG^*}(s)$ and such that $\bL$ is $\dot{w}F$-stable, then the following diagram is commutative:

$$\begin{CD}
\cE(\bG^F, s)           @>J_s>> \cE(C_{\bG^*}(s)^{(\dot{w}_1 F)^*}, 1) \\
@AAR_{\bL^{\dot{w}F}}^{\bG^F}A            @|\\
\cE(\bL^{\dot{w}F}, s)         @>J_s^{\bL}>> \cE(C_{\bL^*}(s)^{(\dot{v}\dot{w}F)^*}, 1)
\end{CD}$$
where $\dot{v} \dot{w} = \dot{w}_1$, and we extend $J_s$ by linearity to generalized characters.

\item Assume $(W, F)$ is irreducible, $(\bG, F)$ is of type $E_8$, and $(C_{\bG^*}(s), (\dot{w}_1 F)^*)$ is of type $E_7 \times A_1$ (respectively, $E_6 \times A_2$, respectively ${^2 E_6} \times {^2 A_2}$).  Let $\bL$ be a Levi of $\bG$ of type $E_7$ (respectively $E_6$, respectively $E_6$) which contains the corresponding component of $C_{\bG^*}(s)$.  Then the following diagram is commutative:
$$\begin{CD}
\cE(\bG^F, s)                  @>J_s>>     \cE(C_{\bG^*}(s)^{(\dot{w}_1 F)^*}, 1) \\
@AAR_{\bL^{\dot{w}_2F}}^{\bG^F}A                  @AAR_{\bL^{*(\dot{w}_2 F)^*}}^{C_{\bG^*}(s)^{(\dot{w}_1 F)^*}}A \\
\cE(\bL^{\dot{w}_2F}, s)^{\bullet}  @>J_s^{\bL}>>     \cE(\bL^{*(\dot{w}_2F)^*}, 1)^{\bullet}
\end{CD}$$
where the superscript $\bullet$ denotes the cuspidal part of the Lusztig series, and $w_2 = 1$ (respectively $1$, respectively the $W_{\bL}$-reduced element of $W_F(s)$ which is in a parabolic subgroup of type $E_7$ of $W$).

\item Given an epimorphism $\varphi: (\bG, F) \rightarrow (\bG_1, F_1)$ such that $\mathrm{ker}(\varphi)$ is a central torus, and semisimple elements $s_1 \in \bG_1^*$, $s = \varphi^*(s_1) \in \bG^*$, the following diagram is commutative:
$$\begin{CD}
\cE(\bG^F, s)                  @>J_s>>     \cE(C_{\bG^*}(s)^{(\varphi(\dot{w}_1) F)^*}, 1) \\
@AA{^\top \varphi}A                           @VV{^\top \varphi^*}V \\
\cE(\bG_1^{F_1}, s_1)           @>J_{s_1}^{\bG_1}>>  \cE(C_{\bG_1^*}(s_1)^{(\dot{w}_1 F_1)^*}, 1),
\end{CD}$$
where ${^\top \varphi}$ denotes the transpose map, ${^\top \varphi}(\chi(g)) = \chi(\varphi(g))$.

\item If $(\bG, F)$ is a direct product, $\bG = \prod_i \bG_i$, then $J_{\prod_i s_i} = \prod_i J_{s_i}^{\bG_i}$.

\end{enumerate}
\end{theorem}

We will need the following result, which follows from Theorem \ref{UniqueJord}.  It has the exact same proof as \cite[Lemma 3.1]{SrVi15}.

\begin{lemma} \label{ConjJ} Let $s, t \in \bT^{*(w_1 F)^*}$, so that $s, t \in C_{\bG^*}(s)^{(\dot{w}_1 F)^*}$, and suppose that there exists $\dot{v} \in \Nr_{\bG^*}(\bT^*)^{(\dot{w}_1 F)^*}$ such that $\dot{v} s \dot{v}^{-1} = t$.  Let $J_s$, $J_t$ be the maps as described in Theorem \ref{UniqueJord}.  If $J_s(\chi) = \psi$, then $J_{t}(\chi) = {^{\dot{v}} \psi}$.
\end{lemma}

We slightly strengthen one part of Theorem \ref{UniqueJord} in the following, based on observations made in Section \ref{unipotent}.

\begin{proposition} \label{UniqueJord+} In Property (2a) of Theorem \ref{UniqueJord}, we may replace ``up to an integer power of $q^{\delta/2}$'' with ``up to an integer power of $q^{\delta}$".
\end{proposition}
\begin{proof}  Note if $J_1(\chi) = \psi$, then property (2a) of Theorem \ref{UniqueJord} is equivalent to the statement that $\omega_{\chi} = \omega_{\psi}$, and our claim is that we have $\omega_{\chi} \beta_{\chi} = \omega_{\psi} \beta_{\psi}$.  First assume $\bG$ is a simple algebraic group of adjoint type.  In all cases other than $\bG$ being of type $E_8$, then from Lemma \ref{Prop2a2} we have $\omega_{\chi} = \omega_{\psi}$ implies $\beta_{\chi} = \beta_{\psi}$.  But also, when $J_1(\chi) = \psi$, property (1) of Theorem \ref{UniqueJord} implies that $\chi$ and $\psi$ must have the same degree (by \cite[Remark 13.24]{dmbook} with $s=1$).  By Lemma \ref{Prop2a2} in the case $\bG$ is type $E_8$, since we have $\omega_{\chi} = \omega_{\psi}$ and $\chi(1)=\psi(1)$, then $\beta_{\chi} = \beta_{\psi}$.  Thus $\omega_{\chi} \beta_{\chi} = \omega_{\psi} \beta_{\psi}$ whenever $\bG$ is a simple algebraic group of adjoint type.  We now reduce to this case by essentially following the arguments in \cite[(1.18)]{Lu76} and \cite[pg. 144]{DiMi90}.

Assume that $\bG = \prod_i \bH_i$ is a direct product with each $\bH_i$ a simple algebraic group of adjoint type.  Each factor $\bH_i$ is either fixed by $F$, or is permuted cyclically by $F$ within a subset of the factors.  By the remark after Lemma \ref{Prop2a}, we may view $\bG^F$ as a direct product of factors with each of the form $\bH_i^{F}$, or is of the form $\bH_i^{F^k}$ when $\bH_i$ is cyclically permuted by $F$ amongst $k$ factors.  By that same remark, and from the case of simple algebraic groups of adjoint type, the statement we desire holds true for either type of factor.  By property (7) of Theorem \ref{UniqueJord}, since the statement holds for each direct factor of $\bG^F$, the statement also holds for $\bG^F$ itself.

Next suppose that $(\bG, F)$ is such that there is an epimorphism $\varphi: (\bG, F) \rightarrow (\bG_1, F_1)$, with kernel a central torus, where the desired statement on the eigenvalues of the Frobenius holds for the unipotent characters of the group $\bG_1^{F_1}$.  By \cite[Proposition 3.15]{Lu78}, the map $\varphi$ induces (through the transpose map ${^\top \varphi}$) a bijection between the unipotent characters of $\bG^F$ and of $\bG_1^{F_1}$, and the dual map $\varphi^*$ yields a bijection between the unipotent characters of $\bG^{*F^*}$ and of $\bG_1^{*F_1^*}$.  The Weyl groups of $\bG$ and $\bG_1$ may be identified (as $W$) via $\varphi$ with $F$ and $F_1$ having the same action.  For any $w \in W$ it follows from \cite[(1.18)]{Lu76} that we may identify the Deligne-Lusztig varieties $X_w$ corresponding to $\bG$ and to $\bG_1$, and that we may identify the actions of $F^{\delta}$ and $F_1^{\delta}$ on $H_c^i(X_w, \overline{\QQ}_{\ell})$ (for any $i \geq 0$).  It follows that the bijections ${^\top \phi}$ and ${^\top \phi^*}$ between sets of unipotent characters preserve the corresponding eigenvalues of Frobenius maps.  We are assuming that if $\chi_1$ is a unipotent character of $\bG_1^{F_1}$ with $J_1^{\bG_1}(\chi_1) = \psi_1$, then $\omega_{\chi_1} \beta_{\chi_1} = \omega_{\psi_1} \beta_{\psi_1}$.  From the commutative diagram in property (6), it follows that we must also have $\omega_{\chi} \beta_{\chi} = \omega_{\psi} \beta_{\psi}$, that is, the statement holds for $(\bG, F)$ if it holds for $(\bG_1, F_1)$.

Finally suppose that $\bG$ is any connected reductive group with connected center, with Frobenius map $F$.  There exists an epimorphism with kernel a central torus from $(\bG, F)$ to its adjoint quotient, say $(\bG_1, F_1)$.  Then $(\bG_1, F_1)$ has the property that $\bG_1$ is a direct product of finite simple algebraic groups of adjoint type, with each direct factor being either fixed by $F_1$, or is in a subset of factors which are permuted cyclically by $F_1$.  By the previous two paragraphs, the desired statement now follows for the arbitrary $\bG^F$.
\end{proof}

\section{Main Results} \label{MainResults}

We may now prove our main result, which essentially states that the action of the Galois group on the Jordan decomposition $(s, \psi)$ of characters is the natural one.  In particular, this image may be calculated with the knowledge of the images of $\hat{s}$ and $\psi$ under the given Galois automorphism.

\begin{theorem} \label{MainThm}  Suppose $Z(\bG)$ is connected, and let $\chi$ be an irreducible complex character of $\bG^F$.  Let $\sigma \in \Gal(\overline{\QQ}/\QQ)$, so $\sigma$ acts on $\QQ(\zeta_m)$, where $m=\ex(\bG^F) = \ex(\bG^{*F^*})$ and $\sigma (\zeta_m) = \zeta_m^r$ with $r \in \ZZ$ and $(r,m) = 1$.

Let $s \in \bT^*$ such that $W_F(s)$ is nonempty, where $\chi \in \cE(\bG^F, s)$ and $J_{s}(\chi) = \psi$.  Then ${^\sigma \chi} \in \cE(\bG^F, s^r)$ and $J_{s^r}({^\sigma \chi}) = {^\sigma \psi}$.  
\end{theorem}
\begin{proof}  It is enough to consider $\sigma \in \Gal(\QQ(\zeta_m)/\QQ)$.  By Lemma \ref{Galseries}, we have ${^\sigma \chi} \in \cE(\bG^F, s^r)$, and we show that in particular $J_{s^r}({^\sigma \chi}) = {^\sigma \psi}$.  Note that ${^\sigma \psi} \in \cE(C_{\bG^*}(s^r)^{(\dot{w}_1 F)^*}, 1) = \cE(C_{\bG^*}(s)^{(\dot{w}_1 F)^*}, 1)$.

Our proof follows the same structure as the proof of \cite[Theorem 4.1]{SrVi15}, replacing complex conjugation by a Galois automorphism.  We prove the claim by induction on the semisimple rank of $\bG$, where the first case is when $\bG = \bT$ is a torus, so that each Lusztig series contains exactly one character, and the statement follows immediately.  Assume now that the statement holds for any group with semisimple rank smaller than $(\bG, F)$, and we prove the statement holds for $(\bG, F)$. 

If $\lambda \in \cE(\bG^F, s^r)$, it follows from Lemma \ref{Galseries} that ${^{\sigma^{-1}} \lambda} \in \cE(\bG^F, s)$.  That is, every character in $\cE(\bG^F, s^r)$ is of the form ${^\sigma \chi}$ for some $\chi \in \cE(\bG^F, s)$.  Given this fact, we have a well-defined map
$$ \mu_{s^r}: \cE(\bG^F, s^r) \longrightarrow \cE(C_{\bG^*}(s^r)^{(\dot{w}_1 F)^*}, 1) = \cE(C_{\bG^*}(s)^{(\dot{w}_1 F)^*}, 1),$$
where $\mu_{s^r}({^\sigma \chi}) = {^\sigma \psi}$ when $J_s(\chi) = \psi$.  We apply the uniqueness described by Theorem \ref{UniqueJord} to show that $\mu_{s^r} = J_{s^r}$, which will give the desired result.  We prove the map $\mu_{s^r}$ satisfies each of the properties listed in Theorem \ref{UniqueJord}, where the induction hypothesis on the semisimple rank is employed only for properties (4) and (5).

For property (1) of Theorem \ref{UniqueJord}, we can apply Lemma \ref{DLGalois}.   Using the fact that $\langle \chi, R_{\bT^{wF}}^{\bG^F}(\hat{s}) \rangle \in \ZZ$, we have
$$\langle \chi, R_{\bT^{wF}}^{\bG^F}(\hat{s}) \rangle = {^\sigma \langle \chi, R_{\bT^{wF}}^{\bG^F}(\hat{s}) \rangle } =  \langle {^\sigma \chi}, {^\sigma R_{\bT^{wF}}^{\bG^F}(\hat{s})} \rangle = \langle {^\sigma \chi}, R_{\bT^{wF}}^{\bG^F}(\widehat{s^r}) \rangle,$$
and similarly,
\begin{align*}
\langle \psi, (-1)^{l(w_1)} R_{\bT^{*(wF)^*}}^{C_{\bG^*}(s)^{(\dot{w}_1 F)^*}} ({\bf 1}) \rangle & =  {^\sigma \langle \psi, (-1)^{l(w_1)} R_{\bT^{*(wF)^*}}^{C_{\bG^*}(s)^{(\dot{w}_1 F)^*}} ({\bf 1}) \rangle} \\
&  = \langle {^\sigma \psi}, (-1)^{l(w_1)} R_{\bT^{*(wF)^*}}^{C_{\bG^*}(s)^{(\dot{w}_1 F)^*}} ({\bf 1}) \rangle.
\end{align*}
Since we have 
$$\langle \chi, R_{\bT^{wF}}^{\bG^F}(\hat{s}) \rangle = \langle \psi, (-1)^{l(w_1)} R_{\bT^{*(wF)^*}}^{C_{\bG^*}(s)^{(\dot{w}_1 F)^*}} ({\bf 1}) \rangle,$$
then it follows we have
$$\langle {^\sigma \chi}, R_{\bT^{wF}}^{\bG^F}(\widehat{s^r}) \rangle = \langle {^\sigma \psi}, (-1)^{l(w_1)} R_{\bT^{*(wF)^*}}^{C_{\bG^*}(s)^{(\dot{w}_1 F)^*}} ({\bf 1}) \rangle,$$
so that $\mu_{s^r}$ satisfies (1).

For property (2), we take $s=1$ and assume $\chi \in \cE(\bG^F, 1)$ is unipotent and $J_1(\chi) = \psi$.  Instead of property (2a), we can use the refined property in Proposition \ref{UniqueJord+}, which is equivalent to $\omega_{\chi} \beta_{\chi} = \omega_{\psi} \beta_{\psi}$, and so $\sigma(\omega_{\chi} \beta_{\chi}) = \sigma(\omega_{\psi} \beta_{\psi})$.  By Lemma \ref{Prop2a1}, we have $\sigma(\omega_{\chi} \beta_{\chi}) = \omega_{^\sigma \chi} \beta_{^\sigma \chi}$ and $\sigma(\omega_{\psi} \beta_{\psi}) = \omega_{^\sigma \psi} \beta_{^\sigma \psi}$.  Thus $\omega_{^\sigma \chi} \beta_{^\sigma \chi} = \omega_{^\sigma \psi} \beta_{^\sigma \psi}$, and so the property in Proposition \ref{UniqueJord+} holds for the map $\mu_1$.  Now assume that $\chi$ is a constituent of $\Ind_{\bB^F}^{\bG^F}(\mathbf{1})$, that is, $\chi$ is in the principal series, which means $\psi = J_1(\chi)$ is in the principal series for $\bG^{*F^*}$.  By (2b) of Theorem \ref{UniqueJord}, $\chi$ and $\psi$ both correspond to the same character $\kappa$ of the Hecke algebra $\mathcal{H}(\bG^F, \bB^F)$ (identified with $\mathcal{H}(\bG^{*F^*}, \bB^{*F^*})$ as in Section \ref{Princ}).  Note that if $\chi$ and $\psi$ are principal series characters, then so are ${^\sigma \chi}$ and ${^\sigma \psi}$, where $\mu_1({^\sigma \chi}) = {^\sigma \psi}$.  By Lemma \ref{Prop2b}, since $\chi$ and $\psi$ both correspond to the character $\kappa$ of the Hecke algebra, then ${^\sigma \chi}$ and ${^\sigma \psi}$ both correspond to the character $\sigma \circ \kappa \circ \sigma^{-1}$ of the Hecke algebra.  (Technically, we must replace $\sigma$ by any extension of $\sigma$ to $\mathrm{Aut}(\mathbb{C}/\mathbb{Q})$ for $\sigma \circ \kappa \circ \sigma^{-1}$ to be well-defined, although it is immediate that this is independent of the choice of extension.)  It follows that property (2b) also holds for $\mu_1$.

For property (3), let $z \in Z(\bG^{*F^*})$, and recall $r$ satisfies $(r, m) = 1$ and $\sigma(\zeta_m) = \zeta_m^r$.  We must show $\mu_{s^r}({^\sigma \chi}) = \mu_{s^r z}({^\sigma \chi} \otimes \hat{z})$.  Let $k \in \ZZ$ such that $rk = 1$ mod $m$.  If $J_s(\chi) = \psi$, then $J_{sz^k}(\chi \otimes \widehat{z^k}) = \psi$, while $\mu_{s^r}({^\sigma \chi}) = {^\sigma \psi}$.  Since $(sz^k)^r = s^r z$ and ${^\sigma (\widehat{z^k})} = \hat{z}^{kr} = \hat{z}$, then by definition we have
$$ \mu_{s^r z} ({^\sigma \chi} \otimes \hat{z}) = {^\sigma J_{sz^k}}( \chi \otimes \widehat{z^k}) = {^\sigma \psi},$$
and property (3) for $\mu_{s^r}$ follows.

For property (4), since the Levi subgroup $\bL$ has semisimple rank strictly smaller than $\bG$, we may apply the induction hypothesis.  So for any $\xi \in \cE(\bL^{\dot{w}F}, s)$, if $J_s^{\bL}(\xi) = \psi$, then $J_{s^r}^{\bL}({^\sigma \xi}) = {^\sigma \psi}$.  Also, if $R_{\bL^{\dot{w}F}}^{\bG^F}(\xi) = \chi$, then $R_{\bL^{\dot{w}F}}^{\bG^F}({^\sigma \xi}) = {^\sigma \chi}$ by Lemma \ref{DLGalois}.  It follows that the diagram in property (4) commutes when we replace $J_s$ with $\mu_{s^r}$ in the top row and $J_s^{\bL}$ with $J_{s^r}^{\bL}$ in the bottom row, as desired.

The proof that the map $\mu_{s^r}$ satisfies property (5) is very similar to the proof for (4) above.  We may again apply the induction hypothesis to the Levi subgroup $\bL$, and so if $\xi \in \cE(\bL^{\dot{w}_2 F}, s)^{\bullet}$ and $J_s^{\bL}(\xi) = \lambda \in \cE(\bL^{*(\dot{w}_2F)^*}, 1)^{\bullet}$, then $J_{s^r}^{\bL}({^\sigma \xi}) = {^\sigma \lambda}$.  If $R_{\bL^{\dot{w}_2 F}}^{\bG^F}(\xi) = \chi$, then $R_{\bL^{\dot{w}_2 F}}^{\bG^F}({^\sigma \xi}) = {^\sigma \chi}$, and if 
$$ R_{\bL^{*(\dot{w}_2 F)^*}}^{C_{\bG^*}(s)^{(\dot{w}_1 F)^*}}(\lambda) = \psi, \quad \text{ then } \quad R_{\bL^{*(\dot{w}_2 F)^*}}^{C_{\bG^*}(s)^{(\dot{w}_1 F)^*}}({^\sigma \lambda}) = {^\sigma \psi}.$$
We know $J_s(\chi) = \psi$, and thus $\mu_{s^r}({^\sigma \chi}) = {^\sigma \psi}$.  The diagram in property (5) therefore commutes when $J_{s^r}^{\bL}$ is on the bottom row and $\mu_{s^r}$ is on the top row, and it follows that property (5) holds for $\mu_{s^r}$.

For property (6), let $\varphi: (\bG, F) \rightarrow (\bG_1, F_1)$ be an epimorphism with kernel a central torus, and $\varphi^*: (\bG_1^*, F_1^*) \rightarrow (\bG^*, F^*)$ the induced map between the dual groups.  Let $\chi_1 \in \cE(\bG_1^{F_1}, s_1)$, ${^\top \varphi}(\chi_1) = \chi \in \cE(\bG^F, s)$, with $J_s(\chi) = \psi$ and $J_{s_1}(\chi_1) = \psi_1$.  By applying property (6) to $J_s$ and $J_{s_1}$, we have ${^\top \varphi^*}(\psi) = \psi_1$.  We have ${^\sigma \psi_1} = {^\sigma ({^\top \varphi^*}(\psi))} =  {^\sigma \psi(\varphi^*)}$.  Thus, ${^\top \varphi^*}({^\sigma \psi}) = {^\sigma \psi_1}$.  We also have $\varphi(s^r) = s_1^r$, and ${^\top \varphi}({^\sigma \chi_1}) = {^\sigma \chi_1}(\varphi) = {^\sigma \chi}$.  Now, the diagram from property (6) is commutative when we have $\mu_{s_1^r}$ in the bottom row and $\mu_{s^r}$ in the top row, as desired.

For the final property (7), suppose $\chi = \prod_i \chi_i$ and $J_s(\chi) = \psi = \prod_i \psi_i$, with $s = \prod_i s_i$ and $J_{s_i}(\chi_i) = \psi_i$.  Then 
$$\mu_{\prod_i s_i^r}({^\sigma \chi)} = \mu_{s^r}({^\sigma \chi}) = {^\sigma \psi} = \prod_i {^\sigma \psi_i} =  \prod_i \mu_{s^r} ({^\sigma \chi_i}).$$
Since all properties from Theorem \ref{UniqueJord} hold for the map $\mu_{s^r}$, we must have $\mu_{s^r} = J_{s^r}$, and it follows that if $J_s(\chi) = \psi$, then $J_{s^r}({^\sigma \chi}) = {^\sigma \psi}$.
\end{proof}

The following is our main application of Theorem \ref{MainThm}, which allows us to reduce the problem of computing the field of character values of an irreducible character of $\bG^F$ to the question of conjugacy of powers of semisimple elements, and the computation of the actions of group and Galois automorphisms on unipotent characters.

\begin{corollary} \label{MainCor}  Let $L$ be any subfield of $\overline{\QQ}_{\ell}$ (or of $\CC$ if these are identified), let $m = \ex(\bG^F) = \ex(\bG^{*F^*})$, and $K = \QQ(\zeta_m) \cap L$.  For any $\sigma \in \Gal(\QQ(\zeta_m)/K)$, let $r_{\sigma}  \in \ZZ$ such that $(r_{\sigma}, m) = 1$ and $\sigma(\zeta_m) = \zeta_m^{r_{\sigma}}$.

Suppose $Z(\bG)$ is connected, and let $\chi$ be an irreducible character of $\bG^F$, where $s_0 \in \bG^{*F^*}$ is semisimple, $\chi \in \cE(\bG^F, s_0)$, and $J_{s_0}(\chi) = \nu$.  Then $\QQ(\chi) \subseteq L$ if and only if the following hold for every $\sigma \in \Gal(\QQ(\zeta_m)/K)$:
\begin{enumerate}
\item[(i)] The element $s_0$ is $\bG^{*F^*}$-conjugate to $s_0^{r_{\sigma}}$.
\item[(ii)] If $h_0 \in \bG^{*F^*}$ satisfies $h_0 s_0 h_0^{-1} = s_0^{r_{\sigma}}$, then ${^{h_0} \nu} = {^\sigma \nu}$.
\end{enumerate}
\end{corollary}
\begin{proof} In the notation of Theorem \ref{MainThm}, we suppose that $\chi \in \cE(\bG^F, s)$ with $s \in \bT^*$ and $W_F(s)$ nonempty (so $s$ corresponds to $s_0$).  First, consider some $h \in \bG^{*(\dot{w}_1 F)^*}$ which normalizes $C_{\bG^*}(s)^{(\dot{w}_1 F)^*}$.  The automorphism of $C_{\bG^*}(s)^{(\dot{w}_1 F)^*}$ given by conjugation by $h$ permutes the unipotent characters $\cE(C_{\bG^*}(s)^{(\dot{w}_1 F)^*}, 1)$, by \cite[(1.27)]{BrMaMi93} for example.  If $hsh^{-1} = s^r$ for some $r \in \ZZ$ such that $(r,m) = 1$, then $h$ normalizes $C_{\bG^*}(s)^{(\dot{w}_1F)^*} = C_{\bG^*}(s^r)^{(\dot{w}_1F)^*}$.  Any other element $h_1 \in \bG^{*(\dot{w}_1 F)^*}$ satisfying $h_1 s h_1^{-1} = s^r$ must have the property that $h_1 \in hC_{\bG^*}(s)^{(\dot{w}_1F)^*}$.  This implies that the action on $\cE(C_{\bG^*}(s)^{(\dot{w}_1 F)^*}, 1)$ given by conjugation by an element $h$ such that $h s h^{-1} = s^r$, is independent of the choice of $h$.

Suppose that $\QQ(\chi) \subseteq L$, which is equivalent to ${^\sigma \chi} = \chi$ for each $\sigma \in \Gal(\QQ(\zeta_m)/K)$ since $\QQ(\chi) \subseteq \QQ(\zeta_m)$.  Then for each $\sigma$, $\chi = {^\sigma \chi} \in \cE(\bG^F, s^{r_{\sigma}})$ by Lemma \ref{Galseries}.  So $\cE(\bG^F, s) = \cE(\bG^F, s^{r_{\sigma}})$, which is equivalent to condition (i).  Since $s, s^{r_{\sigma}} \in \bT^{*(w_1 F)^*}$, then by \cite[Lemma 2.2]{SrVi15} we have $\dot{v} s \dot{v}^{-1} = s^{r_{\sigma}}$ for some $\dot{v} \in \mathrm{N}_{\bG^*}(\bT^*)^{(\dot{w}_1 F)^*}$.  By Lemma \ref{ConjJ}, if $J_s(\chi) = \psi$, then we have $J_{^{\dot{v}} s}(\chi) = {^{\dot{v}} \psi}$.  Since $J_{s^{r_{\sigma}}}(\chi) = {^\sigma \psi}$ and ${^{\dot{v}} s} = s^{r_{\sigma}}$, then ${^{\dot{v}} \psi} = {^\sigma \psi}$.  By the previous paragraph, we also then have ${^h \psi} = {^\sigma \psi}$, which gives condition (ii).

Conversely, suppose that conditions (i) and (ii) hold for each $\sigma$.  That is, for each $\sigma$ we have $h s h^{-1} = s^{r_{\sigma}}$ for some $h \in \bG^{*(\dot{w}_1 F)^*}$, and if $J_{s}(\chi) = \psi$, then ${^h \psi} = {^\sigma \psi}$.  Again by \cite[Lemma 2.2]{SrVi15}, there is some $\dot{v} \in \mathrm{N}_{\bG^*}(\bT^*)^{(\dot{w}_1 F)^*}$ such that ${^{\dot{v}} s} = s^{r_{\sigma}}$, and so also ${^{\dot{v}} \psi} = {^\sigma \psi}$.  We have $J_{s^{r_{\sigma}}}({^\sigma \chi}) = {^\sigma \psi}$ by Theorem \ref{MainThm}, and $J_{^{\dot{v}} s}(\chi) = {^{\dot{v}} \psi}$ by Lemma \ref{ConjJ}.  Now $J_{s^{r_{\sigma}}}({^\sigma \chi}) =  J_{s^{r_{\sigma}}} (\chi)$, and thus ${^\sigma \chi} = \chi$.  Since this holds for every $\sigma \in \Gal(\QQ(\zeta_m)/K)$, we have $\QQ(\chi) \subseteq K \subseteq L$.
\end{proof} 

We conclude with a criterion for an irreducible character $\chi$ of $\bG^F$ to be rational-valued.  Recall that an element $g$ of a finite group $G$ is \emph{rational} if, for every $r \in \ZZ$ such that $(|g|, r)=1$, $g$ and $g^r$ are conjugate in $G$, see \cite[Section 5]{NaTi08} for example.  It follows from Lemma \ref{Galseries} that if $\chi$ has Jordan decomposition $(s_0, \nu)$, and $s_0$ is not rational in $\bG^{*F^*}$, then $\chi$ is not rational-valued.  The following gives a partial converse to this statement.  Note that there are many unipotent characters which are rational-valued \cite{Lu02} and many which are invariant under group automorphisms \cite[Proposition 3.7]{Ma07}.

\begin{corollary} \label{LastCor}  Suppose $Z(\bG)$ is connected, and $\chi$ is an irreducible character of $\bG^F$ with Jordan decomposition $(s_0, \nu)$.  If $s_0$ is rational in $\bG^{*F^*}$, and $\nu$ is both rational-valued and invariant under the automorphism group of $C_{\bG^*}(s_0)^{F^*}$, then $\chi$ is rational-valued.
\end{corollary}
\begin{proof}  If $s_0$ is rational, then for every $r \in \ZZ$ such that $(|s_0|, r)=1$, we have $hs_0h^{-1} = s_0^r$ for some $h \in \bG^{*F^*}$.  For any $\sigma \in \Gal(\overline{\QQ}/\QQ)$, we have ${^\sigma \nu} = \nu$ since $\nu$ is rational-valued, and since $\nu$ is invariant under automorphisms of $C_{\bG^*}(s_0)^{F^*}$, which includes conjugation by elements in $\bG^{*F^*}$, we have ${^h \nu} = \nu = {^\sigma \nu}$.  It follows from Corollary \ref{MainCor} that $\chi$ is rational-valued.
\end{proof}

\noindent {\bf Remark. }  Through conversations with Jay Taylor, it appears that the uniqueness statement for the Jordan decomposition map in Theorem \ref{UniqueJord} might be generalized to Lusztig series $\cE(\bG^F, s)$ such that $C_{\bG^*}(s)$ is connected (or, more generally, if $C_{\bG^*}(s_0)^{F^*} = (C_{\bG^*}(s_0)^{\circ})^{F^*}$, as suggested by Gunter Malle), while $Z(\bG)$ is not necessarily connected.  If this holds, then the results of this section immediately generalize to this situation.

\end{document}